\documentclass[1p]{elsarticle}
\usepackage{amsmath, amssymb,amsthm,stmaryrd}
\usepackage{graphicx}

\journal{Journal of Combinatorial Theory, Series B}

\newtheorem{theorem}{Theorem}

\newtheorem{lemma}[theorem]{Lemma}

\newtheorem{proposition}[theorem]{Proposition}
\newtheorem{corollary}[theorem]{Corollary}
\theoremstyle{definition}
\newtheorem{problem}{Problem}
\newtheorem{definition}{Definition}
\theoremstyle{remark}

\newcommand{\shm}{\,\triangledown\,}
\newcommand{\shtm}{\,\widetilde{\triangledown}\,}

\newcommand{\mshtm}{\,\mathchoice{\widetilde{\rotatebox[origin=c]{90}{\small$\trianglelefteqslant$}}}{\widetilde{\rotatebox[origin=c]{90}{\small$\trianglelefteqslant$}}}{\widetilde{\rotatebox[origin=c]{90}{$\scriptscriptstyle\trianglelefteqslant$}}}{\widetilde{\rotatebox[origin=c]{90}{$\scriptscriptstyle\trianglelefteqslant$}}}\,}
\newcommand{\rdens}[1]{\,{\nabla}_{#1}}
\newcommand{\trdens}[1]{\,\widetilde{\nabla}_{#1}}
\newcommand{\mtrdens}[1]{\,\widetilde{\nabla}\kern-2.8pt\raisebox{1.6pt}{$\scriptstyle /$}_{#1}}
\newcommand{\bbbn}{\mathbb{N}}

\newcommand{\arb}{{\rm Arb}}
\def\md{\Delta^{\rm -}}

 \def\vG{\vec{G}}
 \def\vH{\vec{H}}

\begin{document}
\begin{frontmatter}
\title{Colouring Edges with many Colours in Cycles}
\author{J.~Ne\v set\v ril\fnref{fn1}}
\address{Department of Applied Mathematics and Institute of Theoretical Computer Science (ITI)\\
  Charles University\\
  Malostransk\' e n\' am.25, 11800 Praha 1, Czech Republic}
 \ead{nesetril@kam.ms.mff.cuni.cz}
\fntext[fn1]{Supported by grant 1M0405 of the Czech Ministry of Education}
\author{P.~Ossona de Mendez
\fnref{fn2}}
\address{
Centre d'Analyse et de Math\'ematiques Sociales (CNRS, UMR 8557)\\
  190-198 avenue de France, 75013 Paris, France}
 \ead{pom@ehess.fr}
\fntext[fn2]{Partially supported by the Academia Sinica}
\author{X.~Zhu}
\address{Department of  Mathematics\\
Zhejiang Normal University, China}
\ead{xudingzhu@gmail.com}

\begin{abstract}
The arboricity of a graph $G$ is the minimum number of colours needed to colour the edges of
$G$ so that every cycle gets at least two colours.
 Given  a positive
integer $p$, we define the generalized $p$-arboricity $\arb_p(G)$ of a
graph $G$ as the minimum number of colours needed to colour the
edges of a multigraph $G$ in such a way that every cycle $C$ gets
at least $\min(|C|,p+1)$ colours. In the particular case where $G$
has girth at least $p+1$, $\arb_p(G)$ is the minimum size of a
partition of the edge set of $G$ such that the union of any $p$
parts induce a forest. If we require further that the edge colouring be proper, i.e., adjacent edges
receive distinct colours, then the minimum number of  colours needed is the
generalized $p$-acyclic edge chromatic number of $G$. In this paper,  we relate the generalized $p$-acyclic edge chromatic
numbers and the generalized $p$-arboricities  of a graph $G$ to the
density of the multigraphs having a shallow subdivision as a subgraph of $G$.
\end{abstract}

\begin{keyword}
graph \sep colouring \sep arboricity
\MSC{05C15} {Colouring of graphs and hypergraphs}
\end{keyword}
\end{frontmatter}
\section{Introduction}
In this paper, we consider the following problem: given a graph
$G$, how many colours do we need to colour the edges of $G$ in such
a way that every cycle gets ``many'' colours?
 Of course, the answer to this question depends on the precise meaning of ``many''.
If we require that  each cycle $\gamma$ of length $l$ of $G$ gets $l$
colours, i.e.,  every cycle is rainbow, then the minimum number of
colours needed is equal to the maximum   size  of a block of $G$,
as two edges of $G$ belong to a common cycle if and only if they
belong to the same  block.
If we require that every cycle gets at least $2$ colours, i.e.,
every colour class induces a forest, then the minimum number of
colours needed is the {\em arboricity} $\arb(G)$ of $G$, and its
determination is solved by the well-known Nash-Williams' theorem we recall now.

Denote by $V(G)$ and $E(G)$ the vertex set and the edge set of
$G$. Also denote by $|G|=|V(G)|$ (resp. $\|G\|=|E(G)|$) the {\em
order} of $G$ (resp.  {\em size}). For $A\subseteq V(G)$ denote by
$G[A]$ the subgraph of $G$ induced by $A$. By Nash-Williams'
theorem \cite{arbor,NW64}, the arboricity of a graph $G$ is given by
the formula:
\begin{equation}
\arb(G)=\max_{A\subseteq V(G), |A|>1}\left\lceil\frac{\|G[A]\|}{|A|-1}\right\rceil.
\end{equation}

Here we consider a generalization of these two extreme cases.
A general form of our problem is captured by the following: 
\begin{quote}
Given an unbounded
non-decreasing  function $f:\bbbn\rightarrow\bbbn$ and an integer
$p$, what is the minimum number $N_f(G,p)$ of colours needed to
colour the edges of a graph $G$ in such a way that each cycle
$\gamma$  gets at least $\min(f(|\gamma|), p+1)$ colours?
\end{quote}
Thus for $p=1$ and $f(n)\geq 2$ we get $N_f(G,p)=\arb(G)$.
For an arbitrary graph $G$,
it is usually difficult to determine $N_f(G,p)$.
Our interest is to find upper bound for $N_f(G,p)$ in terms of other graph parameters,
and upper bound for $N_f(G,p)$ for some nice classes of graphs and/or for some nice special functions $f$.

Many colouring parameters are bounded for proper minor closed classes of graphs.
It is natural to ask for which functions $f$ is $N_f(G, p)$ bounded for any proper minor closed
class $\mathcal C$ of graphs. We shall prove (Lemma~\ref{lem:logtight}) that  if $f(2^{p-1})>p-1$ for some value of $p$ then
there is a (quite small) minor closed class of graphs $\mathcal C$,
such that $N_f(G,p)$ is unbounded. On the other hand, we prove (Corollary~\ref{cor:2}) that  if
$f(x) \le \lceil\log_2 x\rceil$ for all $x$ then $N_f(G,p)$ is not only
bounded on proper minor closed classes of graphs, but actually
bounded on a class $\mathcal C$ if and only if $\mathcal C$ has
{\em bounded expansion} (to be defined in Section~\ref{sec:3}).

Next we consider the special function $f(x)=x$. For this special function, the parameter $N_f(G,p+1)$ is
denoted as $\arb_p(G)$ and is called the {\em generalized
$p$-arboricity} of $G$.
So $\arb_p(G)$ is the number of colours needed if we require that each cycle of $G$ gets at least $p+1$ colours or is rainbow if its length is smaller than $p+1$.
Note that if $p=1$, then $\arb_p(G)$ is
the   arboricity  $\arb(G)$ of $G$.
We shall relate the generalized $p$-arboricities  of  a graph to the density of
its shallow topological minors. Toward this end we define the following notions, which are analogous to those defined in
\cite{POMNI} and \cite{ECM2009}. The main difference is that here we consider multigraphs.

Let $G$ be a multigraph and let $r$ be a half integer. A multigraph $H$ is a {\em shallow topological minor } of $G$ at depth $r$
if a $\leq 2r$-subdivision of $H$ is a subgraph of $G$. We denote by $G\mshtm r$ the class of the multigraphs which
 are shallow topological minors of $G$ at depth $r$.
Hence we have
$$G\in G\mshtm 0\subseteq G\mshtm \tfrac{1}{2}\subseteq\dots\subseteq G\mshtm r\subseteq\dots.$$
Notice that the class $G\mshtm 0$ is exactly the monotone closure of $G$, that is the class of all the subgraphs of $G$.

We denote by $\mtrdens{r}(G)$ the maximum density of a graph in $G\mshtm r$, that is:
\begin{equation}
\mtrdens{r}(G)=\max_{H\in G\mshtm r}\frac{\|H\|}{|H|}.
\end{equation}

In this paper, we will give 
lower and upper bounds for $\arb_p(G)$ based on $\mtrdens{\frac{p-1}{2}}(G))$.
For $p=1$, notice that it is an easy consequence of Nash-Williams' Theorem that $\lceil\mtrdens{0}(G)\rceil\leq \arb_1(G)\leq 2\lceil\mtrdens{0}(G)\rceil$. 
In this paper,
we shall show (Theorem~\ref{thm:ecolor}) that for any positive integer $p$,  there is a polynomial $P_p$ such that for any graph $G$,
\begin{equation}
\label{eq:arbp}
\left( \mtrdens{\frac{p-1}{2}}(G)\right)^{1/p} \leq  \arb_p(G)  \leq P_p(\mtrdens{\frac{p-1}{2}}(G)).
\end{equation}

The paper is organized as follows: In Section~\ref{sec:2}, we consider the key 
case of graphs with bounded tree-depth. In particular, we establish that 
if $f(2^{p-1})>p-1$ for some value of $p$ then
there is a minor closed class of graphs $\mathcal C$ (namely the class of graphs with tree-depth at most $p$) such that $N_f(G,p)$ is unbounded.
In Section~\ref{sec:3},
we prove that if, for some unbounded non-decreasing function $f$ and each fixed integer $p$, the value $N_f(G,p)$ is bounded for graphs in a class $\mathcal C$, then the class $\mathcal C$ has bounded expansion. We prove also that, conversely, if $\mathcal C$ has bounded expansion and $f_0(x)=\lceil\log_2 x\rceil$ then $\sup_{G\in\mathcal C}N_{f_0}(G,p)$ is bounded for each integer $p$. In Section~\ref{sec:4}, we establish~\eqref{eq:arbp}, which is the main result of this paper.
For the sake of improving the readability of this paper, the proofs of two difficult lemmas used in Section~\ref{sec:4} are actually postponed to Section~\ref{sec:5}.
In Section~\ref{sec:6} we consider a dual version of the problem.
\section{Longest Cycles and Tree-Depth}
\label{sec:2}
Let us recall some definitions.
The {\em height} of a rooted forest is the maximum number of
vertices in a path from a root to a leaf. The {\em closure} of a
forest $F$ is the graph on $V(F)$ in which $xy$ is an edge if and
only if $x$ is an ancestor of $y$ or $y$ is an ancestor of $x$.
 The {\em tree-depth}
${\rm td}(G)$ of a graph $G$ is the minimum height of a rooted
forest $F$ such that $G$ is a subgraph of its closure.

In this section we establish how the concept of tree-depth introduced in \cite{Taxi_tdepth}
is related to the length of the longest cycle of a graph.
It will follow  that if $f(2^{p-1})>p-1$ for some value of $p$ then
there is a minor closed class of graphs $\mathcal C$ on which $N_f(G,p)$ is unbounded.

\begin{lemma}
\label{lem:logtight} Let $p$ be an integer such that the function $f$ satisfies $f(2^{p-1})>p-1$. Let
$\mathcal C$ be the class of graphs with tree-depth at most $p$.
Then $N_f(G,p)$ is unbounded on $\mathcal C$.
\end{lemma}
\begin{proof}
Let $N$ be an arbitrarily large integer. Let $G$ be the closure of
a rooted complete $q$-ary tree $Y$ of  height $p$, where
$q=N^{\binom{p}{2}}+1$. Let $r$ be the root of $Y$. Given a leaf
$v$ of $Y$, there are $N^{\binom{p}{2}}$ ways to colour the edges
of the subgraph of $G$ induced by $v$ and its ancestors with $N$
 colours, and let $\phi(v)\in\{1,\dots,N^{\binom{p}{2}}\}$ be
the encoding of this colouration corresponding to the leaf vertex
$v$. For non leaf vertices $v$ we define $\phi(v)$ by induction on the
descending height as the majority value of $\phi(x)$ among the
children of $v$. The root $r$ has at least $\lceil
q/N^{\binom{p}{2}}\rceil=2$ sons $v$ with $\phi(v)=\phi(r)$.
Inductively, the root $r$ is the root of a complete binary subtree
$Y'$ of $Y$, all vertices of which have the same $\phi$-value as
$r$. The closure of $Y'$ contains a cycle $\gamma$ of length
$2^{p-1}$, and this cycle gets
at most $p-1$ colours (see
Fig~\ref{fig:R3}). As $\min(f(|\gamma|),p)=p$ we conclude that
$N_f(G,p)>N$ and thus $N_f(G,p)$ is unbounded.
\end{proof}
\begin{figure}[h]
    \centering
        \includegraphics[width=0.70\textwidth]{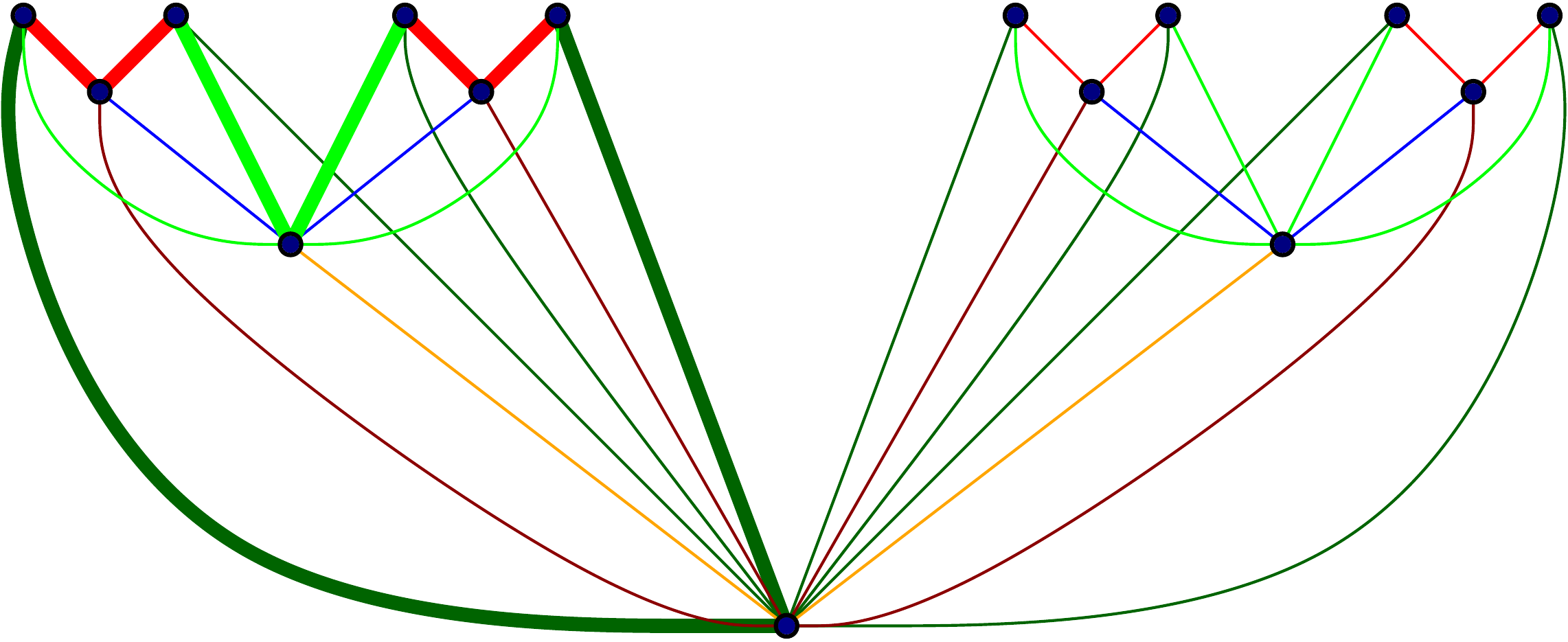}
    \label{fig:R3}
    \caption{The closure of $Y'$ contains a cycle $\gamma$ of length $2^{p-1}$, and this cycle gets exactly $p-1$ colours}
\end{figure}

Remark that the proof of Lemma~\ref{lem:logtight} is a variant of an old trick of R. Goldblatt. Lemma~\ref{lem:logtight} shows that we cannot  expect $N_f(G,p)$
to be bounded on proper minor closed classes of graphs if
$f(x)>\lceil\log_2 x\rceil$. In Section~\ref{sec:3}, we
shall show that if $f(x) \le \lceil\log_2 x\rceil$ then $N_f(G,p)$
are not only bounded on proper minor closed classes of graphs, but
actually bounded on a class $\mathcal C$ if and only if $\mathcal
C$ has bounded expansion. This provides yet another characterization of this robust notion.


We now prove that the connection with tree-depth shown in Lemma~\ref{lem:logtight} is actually deeper in the sense that a $2$-connected graph has no long cycles if and only if it has a small tree-depth. (Note that $2$-connectivity has to be assumed.)

\begin{lemma}
\label{lem:tdL}
Let $G$ be $2$-connected graph and let $L$ be the maximum length of a cycle in $G$. Then
$$1+\lceil\log_2 L\rceil\leq {\rm td}(G)\leq \binom{L-1}{2}+2.$$
\end{lemma}
\begin{proof}
The first inequality is a consequence of the monotonicity of tree-depth and the exact values of tree-depth for cycles: ${\rm td}(C_n)=1+\lceil\log_2 n \rceil$.

The remaining of the proof will concern the second inequality.

Consider a Depth-First Search tree $Y$ of $G$ and let $r$ be the root of $Y$.
Let $h={\rm height}(Y)$.
For a vertex $x$ let ${\rm level}(x)$ be the height of $x$ in $Y$.
The rooted tree $Y$ naturally defines a partial order $\preceq$ by $x\preceq y$ if $x$ belongs to the tree-path from $r$ to $y$.
A
basic property of DFS-trees is that two adjacent vertices are always comparable with respect to $\preceq$ (DFS-trees have no cross edges), so that $G\subseteq{\rm Clos}(Y)$ thus ${\rm td}(G)\leq h$.

For a vertex $x$, let ${\rm low}(x)$ be the smallest vertex $y$ (with respect to $\preceq$) which is adjacent
by an edge not in $Y$
to a vertex $z\succeq x$. Notice
that such a vertex exists as $G$ is $2$-connected and that for $x\neq r$ it holds ${\rm low}(x)\prec x$.
Moreover, if ${\rm low}(x)\neq r$, the fact that ${\rm low}(x)$ is not a cut-vertex of $G$ implies that there exists $y$ such that ${\rm low}(y)\prec {\rm low}(x)\prec y\prec x$.

Let $P=(v_1=r,\dots,v_h)$ be a longest root-to-leaf path of $Y$.
We inductively define indexes $a_1,b_1,\dots,a_k,b_k\in\{1,\dots,h\}$ as follows: let $a_1=h$ and let $v_{b_1}={\rm low}(v_{a_1})$.
For $i\geq 1$, if $b_i\neq 1$ then let $v_{b_{i+1}}$ be the minimum ${\rm low}$ of a vertex $z$ such that $v_{b_i}\prec z\prec v_{a_i}$, let
$v_{a_{i+1}}$ be such that  $v_{b_i}\prec v_{a_{i+1}}\prec v_{a_i}$ and ${\rm low}(a_{i+1})=b_{i+1}$, and let $e_i$ be a non-tree edge linking some $w\succeq v_{a_i}$ and $v_{b_i}$.
This process stops at some value $k$ such that $b_k=1$. Notice that $a_{i+2}\preceq b_i$ for $1\leq i\leq k-2$.

Let $\gamma_i$ be the fundamental cycle of $e_i$ and let $\gamma$ be the symmetric difference of all the $\gamma_i$'s. Then  $\gamma$ is a cycle and each edge of $P\cup\{e_1,\dots,e_k\}$ belongs to $2$ of the cycles $\gamma_1,\dots,\gamma_k,\gamma$. Hence
$2(h+k-1)\leq |\gamma_1|+\dots+|\gamma_k|+|\gamma|\leq (k+1)L$, i.e. $h\leq \frac{k+1}{2}(L-2)+2$ (see Fig.~\ref{fig:tdL}).

Moreover, $\gamma$ contains at least two tree edges (either $k=1$ and $\gamma=\gamma_1$ or $k>1$ and then $a_2<a_1$ and $b_k<b_{k-1}$) thus $L\geq |\gamma|\geq k+2$ hence $k\leq L-2$.
Altogether, we obtain $h\leq \binom{L-1}{2}+2$.
\end{proof}

\begin{figure}[h!t]
	\centering
		\includegraphics[width=0.25\textheight]{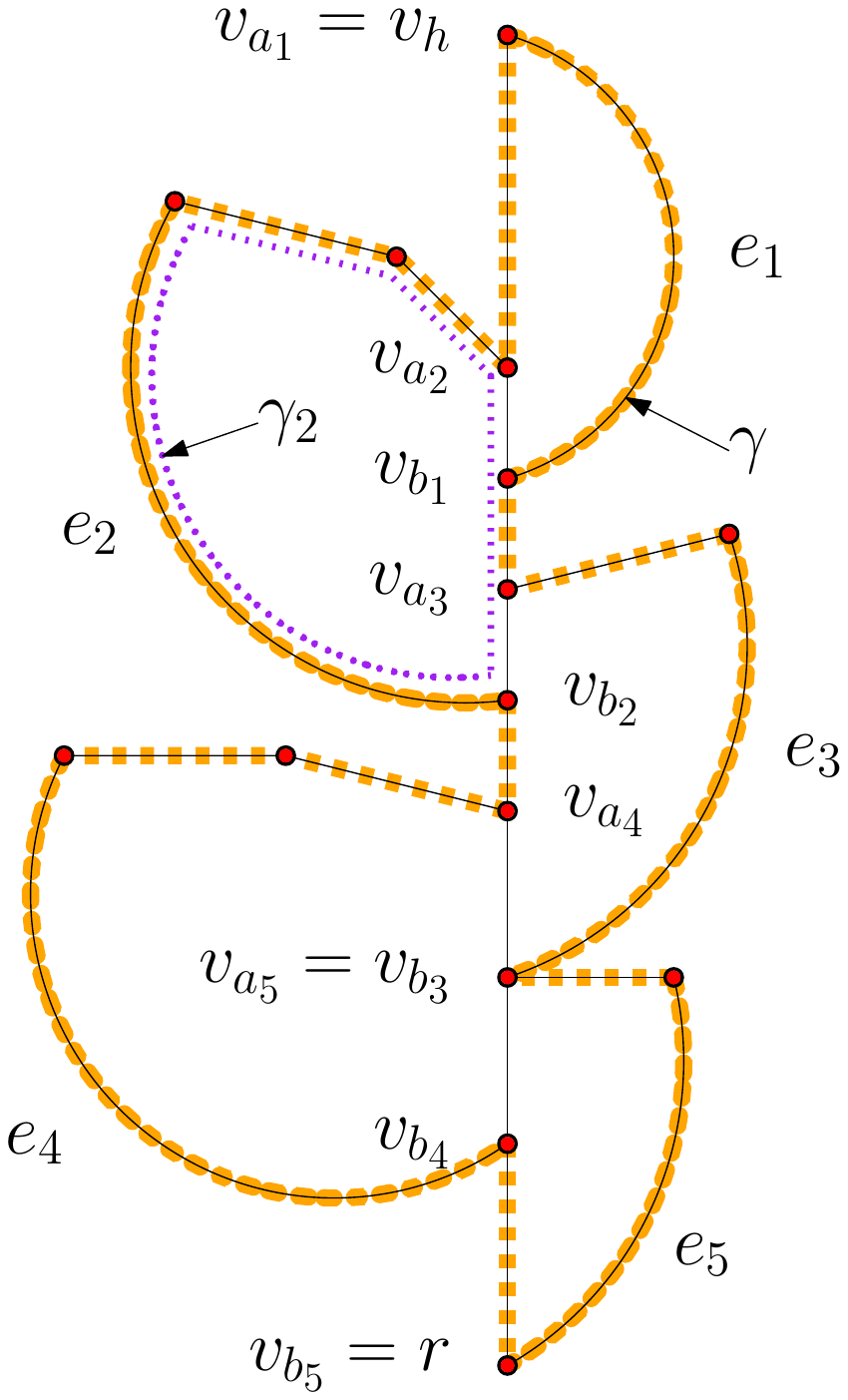}
	\caption{Illustration of the proof of Lemma~\ref{lem:tdL}. Cotree edge $e_i$ links $v_{a_i}$ (or one of its 
 symmetric difference of the $\gamma_i$'s}
	\label{fig:tdL}
\end{figure}

It is not clear whether the quadratic bound of Lemma~\ref{lem:tdL} is tight. We propose the following problem (compare with \cite{Kenkre2007} where the chromatic number of a graph is bounded using the length of the largest odd cycle).
\begin{problem}
For a $2$-connected graph $G$, denote by $L(G)$ the length of the longest cycle in $G$.

Does there exist a constant $C$ such that for every $2$-connected graph the following inequality holds:
$${\rm td}(G)\leq C\, L(G)?$$
\end{problem}

\section{Classes of graphs with bounded expansion}
\label{sec:3}

Classes with bounded expansion have been introduced in \cite{ICGT05,Taxi_stoc06,POMNI} and are based on the boundedness of graph invariants similar to $\mtrdens{r}(G)$.

We denote \cite{POMNI,ECM2009} by $G\shm r$ (resp. $G\shtm r$) the class of the simple graphs which
 are shallow minors (resp.
 simple
 shallow topological minors) of $G$ at depth $r$,
and we denote by $\rdens{r}(G)$ (resp. $\trdens{r}(G)$) the maximum density of a graph in
$G\shm r$ (resp. in $G\shtm r$), that is:
\begin{equation}
\rdens{r}(G)=\max_{H\in G\shm r}\frac{\|H\|}{|H|}\qquad
\trdens{r}(G)=\max_{H\in G\shtm r}\frac{\|H\|}{|H|}.
\end{equation}
Notice that the main difference between the definition of $\mtrdens{r}(G)$ and the one of $\trdens{r}(G)$ stands in the way parallel edges are handled.

A class $\mathcal C$ has {\em bounded expansion} if $\sup_{G\in\mathcal C}\rdens{r}(G)$ is bounded for each value of $r$.
It is obvious that $\rdens{r}(G)\geq\trdens{r}(G)$.
However, it has been proved by Dvo{\v r}{\'a}k \cite{Dvo2007,Dvov2007} that
for each integer $r$, $\rdens{r}(G)$ is bounded by a polynomial function of $\trdens{r}(G)$. Hence a class $\mathcal C$ has bounded expansion if and only if $\sup_{G\in\mathcal C}\trdens{r}(G)$ is bounded .

\begin{theorem}
Let $f:\bbbn\rightarrow\bbbn$ be an unbounded non-decreasing function with $f(x)\leq x$ and let $g:\bbbn\rightarrow\bbbn$ be defined by
$$g(p)=\max\{i, f(i)\leq p\}.$$

Then for every graph $G$ and every integer $r$ we have:
$$
\rdens{r}(G)\leq N_f(G,2r+1)^{2r+1}g(2r+1)^2.
$$
\end{theorem}
\begin{proof}
Let $N=N_f(G,2r+1)$ and let $c:E(G)\rightarrow [N]$ be a colouring of the edges of $G$ such that each cycle $\gamma$ gets at least $\min(f(|\gamma|), 2r+2)$ colours.
For a subset $I\in\binom{[N]}{2r+1}$ of $2r+1$ colours, let $G_I$ be the subgraph of $G$ whose edges are coloured by colours  in $I$.
Then the maximum length of a cycle of $G_I$ is $g(2r+1)$. According to Lemma~\ref{lem:tdL}, blocks of $G_I$ have tree-depth at most $\binom{g(2r+1)-1}{2}+2$.

Let $K\in G\shm r$ be such that $\|K\|/|K|=\rdens{r}(G)$, let $x_1,\dotsc,x_k$ be the
roots of trees $T_1,\dots,T_k$ of height at most $r$ (corresponding to vertices $h_1,\dotsc,h_k$ of $K$)
and $H\subseteq G[V(T_1)\cup\dots\cup V(T_k)]$ be such that $K\cong H/(E(T_1)\cup\dots\cup E(T_k))$.

If $h_i$ and $h_j$ are adjacent in $K$ then there exists in $H$ a path $P_{i,j}$ of length at most $2r+1$ linking $x_i$ and
 $x_j$. Denote by $A_{i,j}$
 a subset of $2r+1$ colours such that all the edges of $P_{i,j}$ have their colour in $A_{i,j}$.
For $I\in\binom{[N]}{2r+1}$, let $H_I$ be subgraph of $H$ containing the edges of the paths $P_{i,j}$ for which $A_{i,j}=I$.
Let $K_I$ be the corresponding subgraph of $K$. As the blocks of $H_I$ are included
in  blocks of $G_I$, they have tree-depth at most
$\binom{g(2r+1)-1}{2}+2$. As tree-depth is minor monotone, this bound also applies to the blocks of $K_I$.
Observe that a graph of tree depth $k$ has density at most $k-1$, because if all edges are oriented from higher level end vertex to lower level end vertex, then each vertex has out-degree at most $k-1$.
It follows that
$\|K_I\|/|K|\leq \|K_I\|/|K_I|\leq \binom{g(2r+1)-1}{2}+1$. By summing up over the possible choices of $I$ we obtain
$$\|K\|/|K|\leq\binom{N}{2r+1}\left(\binom{g(2r+1)-1}{2}+1\right)\leq N_f(G,2r+1)^{2r+1}g(2r+1)^2.$$
\end{proof}

\begin{corollary}
\label{cor:1}
Let $\mathcal C$ be a class of graph and let $f:\bbbn\rightarrow\bbbn$ be an unbounded non-decreasing function with $f(x)\leq x$.

If $\sup_{G\in\mathcal C}N_f(G,r)<\infty$ for every $r\in\bbbn$ then $\mathcal C$ has bounded expansion.
\end{corollary}

\begin{theorem}
\label{thm:1}
Let $f_0(x)=\lceil\log_2 x\rceil$ and let $r$ be an integer. There exists a polynomial $P_r$ such that for every graph $G$ it holds
$$N_{f_0}(G,r)\leq P_r(\trdens{2^r}(G)).$$
\end{theorem}
\begin{proof}
According to \cite{Zhu2008} there exists for each $p\in\bbbn$ a polynomial $Q_p$ such that every graph $G$ has a vertex colouring $c:V(G)\rightarrow[Q_p(\trdens{2^{p-1}}(G))]$ such that the subgraph induced by any $i\leq p$ colours has tree-depth at most $i$. Let $p=r+1$ and let $P_r(X)=\binom{Q_{r+1}(X)}{2}$.
The graph $G$ admits a vertex-colouring $c:V(G)\rightarrow [Q_{r+1}(\trdens{2^{r}}(G))]$ such that the subgraph induced by any $i\leq r+1$ colours has tree-depth at most $i$.
Colour each edge $\{x,y\}$ of $G$ by the set $\{c(x),c(y)\}$ (hence using $P_{r}(\trdens{2^{r}}(G))$ colours). Let $\gamma$ be a cycle of $G$ with $i\leq r$ colours. By construction, the vertices of $\gamma$ use at most $i+1\leq r+1$ colours hence ${\rm td}(\gamma)\leq i+1$, thus $|\gamma|\leq 2^i$. It follows that every cycle $\gamma$ of $G$ gets at least
$\min(r+1,f_0(|\gamma|)$ colours.
\end{proof}

\begin{corollary}
\label{cor:2}
Let $\mathcal C$ be a class of graphs. Then the following are equivalent:
\begin{enumerate}
	\item There exists non-decreasing unbounded $f:\bbbn\rightarrow\bbbn$ such that
	$$\forall p\in\bbbn,\qquad\sup_{G\in\mathcal C}N_f(G,p)<\infty$$
	\item Let $f_0(x)=\lceil\log_2\,x\rceil$. Then
	$$\forall p\in\bbbn,\qquad\sup_{G\in\mathcal C}N_{f_0}(G,p)<\infty$$
	\item the class $\mathcal C$ has  bounded expansion.
\end{enumerate}
Notice that the second condition is tight, according to Lemma~\ref{lem:logtight}.
\end{corollary}

\section{Bounds on $\arb_p(G)$}
\label{sec:4}

The lower bound $\left( \mtrdens{\frac{p-1}{2}}(G)\right)^{1/p}$ for $\arb_p(G)$ is easy.

\begin{lemma}
\label{lem:lowbound}
Let $G$ be a graph and let $p$ be a positive integer. Then $(\arb_p(G))^p $ is greater than or equal to the maximum arboricity of a multigraph $H$ such that $G$ contains a $\leq (p-1)$-subdivision of $H$, that is,
\begin{equation}
(\arb_p(G))^p \geq\max\{\arb(H), H\in G\mshtm\bigl(\tfrac{p-1}{2}\bigr)\}\ge\mtrdens{\tfrac{p-1}{2}}(G).
\end{equation}
\end{lemma}
\begin{proof}
Let $c$ be an edge colouring of $G$ with a set $J$ of $\arb_p(G)$ colours such that every cycle $C$ of $G$ gets at least $\min(|C|,p+1)$ colours.
Assume that $G$ includes a $\leq (p-1)$-subdivision $S$ of a multigraph $H$.
Colour each edge $e$ of $H$ by the set $X$ of   colours   used by the edges of the path  of length at most $p$ in $G$ corresponding to $e$ in $S$. The total number of colours used by edges of $H$ is at most the number of $\leq p$-subsets of the $J$, which is less than $(\arb_p(G))^p$.
If $(\arb_p(G))^p  < \arb(H)$, then  $H$ has a monochromatic cycle, each edge being coloured by a set $X$ of at most $p$ colours.
Then the corresponding cycle $C$ of $G$ uses at most $|X|<p+1$ colours and has length at least $2|X|$, contradicting the colouring assumption.
Thus $(\arb_p(G))^p \geq \arb(H) $.
\end{proof}

The upper bound is more involved. First, we introduce the admittedly rather technical definition of fraternal completion of oriented multigraphs.

A digraph $\vG$ is {\em fraternally oriented} if $(x,z)\in E(\vG)$ and
$(y,z)\in E(\vG)$ implies $(x,y)\in E(\vG)$ or $(y,x)\in E(\vG)$. This
concept was introduced by Skrien \cite{skrien82} and a
characterization of fraternally oriented digraphs having no symmetrical
arcs has been obtained by Gavril and Urrutia \cite{gavril92}, who also proved that
triangulated graphs and circular arc graphs are all fraternally
orientable graphs.

In the context of multigraphs, this notion may be extended as follows:
\begin{definition}
Let $\vG$ be a directed multigraph and let $a$ be a positive integer. A {\em fraternal completion} of $\vG$ of {\em depth} $a$
is a triple $\mathfrak f=((E_1,\dots,E_a),w, \kappa)$, where
\begin{itemize}
\item $E_1=E(\vG)$ is the arc set of $\vG$; for each $2\leq i\leq a$, $E_i$ is the arc set of a multigraph having $V(\vG)$ as its vertex set; for every $1\leq i< j\leq a$, $E_i\cap E_j=\emptyset$  (although {\bf different} arcs of $E_i$ and $E_j$ may have the same head and tail);
\item for $e\in\bigcup_{1\leq i\leq a}E_i$, the {\em weight} $w(e)$ of $e$ is the integer $i\in [a]$ such that $e\in E_i$;
\item $\kappa:\bigcup_{1<i\leq a}E_i\rightarrow \bigl(\bigcup_{1\leq i\leq a}E_i\bigr)^2$ is such that for every $e\in \bigcup_{1<i\leq a}E_i$ with $\kappa(e)=(f,g)$ we have
\begin{align*}
{\rm tail}(f)&\neq {\rm tail}(g)\\
w(e)&=w(f)+w(g)\\
{\rm tail}(e)&={\rm tail}(f)\\
{\rm head}(e)&={\rm tail}(g)\\
{\rm head}(f)&={\rm head}(g);
\end{align*}
\item conversely, for every $i,j\in\bbbn, f\in E_i$ and $g\in E_j$ such that $i+j\leq a, {\rm tail}(f)\neq {\rm tail}(g)$ and ${\rm head}(f)={\rm head}(g)$ there exists a unique $e\in E_{i+j}$ such that $\kappa(e)\in\{(f,g),(g,f)\}$.
\end{itemize}
\end{definition}
We also define the {\em arc set} $E_\mathfrak f$ of the fraternal completion $\mathfrak f$ by $E_\mathfrak f=\bigcup_{1\leq i\leq a}E_i$ (notice that $E_{\mathfrak f}$ includes no loop) and define $a={\rm depth}(\mathfrak f)$ as the depth of $\mathfrak f$.
A fraternal completion $\mathfrak f'=((E_1',\dots,E_{a'}'),w', \kappa')$ of $\vG$ {\em extends} another fraternal completion $\mathfrak f=((E_1,\dots,E_a),w, \kappa)$ of $\vG$
(or is an {\em extension} of $\mathfrak f$) if
\begin{itemize}
    \item $a'={\rm depth}(\mathfrak f')>a={\rm depth}(\mathfrak f)$,
    \item for every $1\leq i\leq {\rm depth}(\mathfrak f)$ we have $E_i'=E_i$,
    \item the restrictions of $\kappa$ and $\kappa'$ to $E_{\mathfrak f}$ coincide.
\end{itemize}

We now state an easy lemma of fraternal completions:
\begin{lemma}
\label{lem:fratobv}
For every oriented multigraph $\vG$ and every positive integer $a$,
\begin{itemize}
    \item $\vG$ has a unique fraternal completion of depth $1$ defined by $E_1=E(\vG)$,
    \item every fraternal completion $\mathfrak f$ of depth $a$ has an extension of depth $a+1$.
\end{itemize}
\end{lemma}
\begin{proof}
The first item is direct from the definition.

For the second item, let $\mathfrak f=((E_1,\dots,E_a),w, \kappa)$ be a fraternal completion of $\vG$ of depth $a$.
Consider an arbitrary numbering $\nu$ of $E_{\mathfrak f}$.
Define
\begin{equation*}
\begin{split}
E_{a+1}=\{e_{f,g}: (f,g)\in E_{\mathfrak f}^2, \qquad &w(f)+w(g)=a+1, \nu(f)<\nu(g),\\
&{\rm tail}(f)\neq{\rm tail}(g),\text{ and }{\rm head}(f)={\rm head}(g)\},
\end{split}
\end{equation*}
where $e_{f,g}$ is an arc with  ${\rm tail}(e_{f,g}) = {\rm tail}(f)$ and   ${\rm head}(e_{f,g})= {\rm tail}(g)$;
define the mapping $\kappa':\bigcup_{1\leq i\leq a+1} E_i\rightarrow E_{\mathfrak f}$ by
$\kappa'(e)=\kappa(e)$ if $w(e)\leq a$ and $\kappa'(e_{f,g})=(f,g)$ for $e_{f,g}\in E_{a+1}$;
also define $w':E_{\mathfrak f}\cup E_{a+1}$ by $w'(e)=i$ if $e\in E_i$.
Then $\mathfrak f'=((E_1,\dots,E_{a+1},w',\kappa')$ is obviously an extension of $\mathfrak f$ of depth $a+1$.
\end{proof}

Suppose $\mathfrak f=((E_1,\dots,E_a),w, \kappa)$ is a fraternal
completion of $\vG$ of depth $a$.  We associate
 to each arc $e\in E_{\mathfrak f}$ a walk $W(e)$ of $\vG$ defined as follows:
\begin{itemize}
    \item If $w(e)=1$ then $W(e)$ is the walk $e$,
    \item Otherwise, if $\kappa(e)=(f,g)$ then $W(e)$ is the walk $W(f)$ followed by the reverse $\overline{W}(g)$ of the walk $W(g)$, what we denote by $W(e)=W(f)\cdot \overline{W}(g)$.
\end{itemize}
This way, each arc $e$ of $E_{\mathfrak f}$ is associated a walk $W(e)$ in $\vG$ of length $w(e)$ which has the same endpoints as $e$.
An arc $e\in E_{\mathfrak f}$ is called {\em simple} if $W(e)$ is a path.

Let $\vG$ be a directed multigraph, let $\mathfrak f=((E_1,\dots,E_a),w, \kappa)$ be a fraternal completion of $\vG$ of depth $a$.
Let $\prec$ be the partial order on $E_{\mathfrak f}$ defined by transitivity from the conditions
$$\kappa(e)=(f,g)\quad\Longrightarrow\quad f\prec e\text{ and }g\prec e.$$
Notice that if $e\in E_{\mathfrak f}$ and $f\in E_1$ we have $e\succeq f$ if and only if $f$ belongs to the walk $W(e)$.

For $i=1,2,\ldots, a$,  let $\vH_i$ (resp. $\vH_{\leq i}$) be the
multigraphs with vertex set $V(\vG)$ and arc set $E_i$ (resp.
$\bigcup_{j\leq i}E_j$). In particular, $\vG=\vH_1$. For arcs
$e_1,e_2$ of $\vG$ and a fraternal completion $\mathfrak
f=((E_1,\dots,E_a),w,\kappa)$ of $\vG$ of depth $a$ of $\vG$, say
that a pair $(e_1,e_2)\in E_1^2$ is a {\em conflict} if there exists arcs
$f_1\succeq e_1$ and $f_2\succeq e_2$ and a
directed
path of length at most $a$ of $\vH_{\leq a}$ starting with $f_1$
and ending at one of the endpoints of $f_2$ (notice that we allow
$f_1=f_2$).

\begin{lemma}
\label{lem:color}
Assume
$\vG$ is an orientation of $G$, $\mathfrak
f$ is a fraternal completion of $\vG$ of depth $a$.
Assume that  $c:E(\vG)\rightarrow [N]$ is a colouration of the edges of $\vG$ such that for every conflict $(e_1,e_2)$ we have $c(e_1)\neq c(e_2)$.
Then every cycle $\gamma$ gets at least $\min(|\gamma|,a+1)$ colours.
\end{lemma}
\begin{proof}
Assume  for contradiction that there exist in $\vG$ a cycle $\gamma=(v_1,\dots,v_{|\gamma|})$ which gets less than $\min(|\gamma|,a+1)$ colours.
We say a sequence $(e_1,e_2,\ldots, e_{q+1})$ of simple arcs in $E_{\mathfrak f}$ is {\em admissible} if
the $W(e_i)$ are pairwise arc-disjoint and form consecutive subpaths of $\gamma$
and $e_1,\dots,e_q$ form a directed path of $\vH_{\leq a}$.

Choose an admissible sequence $(e_1,e_2,\ldots, e_{q+1})$ in such a way that
$\sum_i w(e_i)$ is maximal and then that $q$ is minimum.
Without loss of generality, we may assume that for $1\leq i\leq q$ we have $e_i=(v_{a_i},v_{a_{i+1}})$  with $1=a_1<a_2<\dots<a_{q+1}\leq |\gamma|$.

First we show that $\sum_{i=1}^{q+1} w(e_i) \geq \min(a+1,|\gamma|)$.
Assume for contradiction that $\sum_{i=1}^{q+1} w(e_i)<\min(a+1,|\gamma|)$. Then $e_{q+1}$ has $v_{a_{q+1}}$ and $v_{a_{q+2}}$ as endpoints with $a_{q+1}<a_{q+2}
< |\gamma|$
(actually we have $a_{q+2}=\sum_i w(e_i)-1$). Let $g$ be the arc of $\vG$ linking $v_{a_{q+2}}$ to the next vertex of $\gamma$. According to the maximality of $\sum_i w(e_i)$,
the sequence $(e_1,\dots,e_{q+1},g)$ is not admissible hence $e_1,\dots,e_{q+1}$ is not a directed path, that is: $e_{q+1}=(v_{a_{q+2}}, v_{a_{q+1}})$.
As $w(e_q)+w(e_{q+1})\leq \sum_i w(e_i)\leq a$ there exists an arc $f$ in $E_{\mathfrak f}$ such that $\kappa(f)=(e_q,e_{q+1})$ (see Fig~\ref{fig:chicolor}). The arc $f$ is clearly simple ($W(f)=W(e_q)\overline{W}(e_{q+1})$ or
$W(f)=W(e_{q+1})\overline{W}(e_{q})$).
If $\kappa(f)=(e_q,e_{q+1})$ then $(e_1,\dots,e_{q-1},f,g)$ is admissible, and $\sum_{i=1}^{q-1}w(e_i)+w(f)+w(g)=\sum_{i=1}^{q+1}w(e_i)+w(g)>\sum_{i=1}^{q+1} w(e_i)$, what contradicts the maximality of $\sum_i w(e_i)$. Otherwise, $\kappa(f)=(e_{q+1},e_q)$. Then, $(e_1,\dots,e_{q-1},f)$ is an admissible sequence such that $\sum_{i=1}^{q-1}w(e_i)+w(f)=\sum_{i=1}^{q+1}w(e_i)$, what contradicts the minimality of $q$ (for given $\sum_i w(e_i)$).

Thus $\bigl|\bigcup_i W(e_i)\bigr|=\sum_{i=1}^{q+1} w(e_i)\geq \min(a+1,|\gamma|)$. By assumption, the cycle $\gamma$ which gets less than $\min(|\gamma|,a+1)$ colours hence there exist
 $f_1,f_2\in\bigcup_i W(e_i)$ such that $c(f_1)=c(f_2)$.
 Let $b_1,b_2$ be such that $f_1\in W(e_{b_1})$ and $f_2\in W(e_{b_2})$. Without loss of generality we assume $b_1\leq b_2$. As $f_1\preceq e_{b_1}, f_2\preceq e_{b_2}$ and as there exists by construction a (maybe empty) directed path of length at most $a$ of $\vH_{\leq a}$ starting with $e_{b_1}$ and ending at one of the endpoints of $e_{b_2}$ we deduce that $(f_1,f_2)$ is a conflict, contradicting the hypothesis that $c(f_1)=c(f_2)$.
\end{proof}

\begin{figure}[h!t]
\begin{center}
    \includegraphics[width=0.75\textwidth]{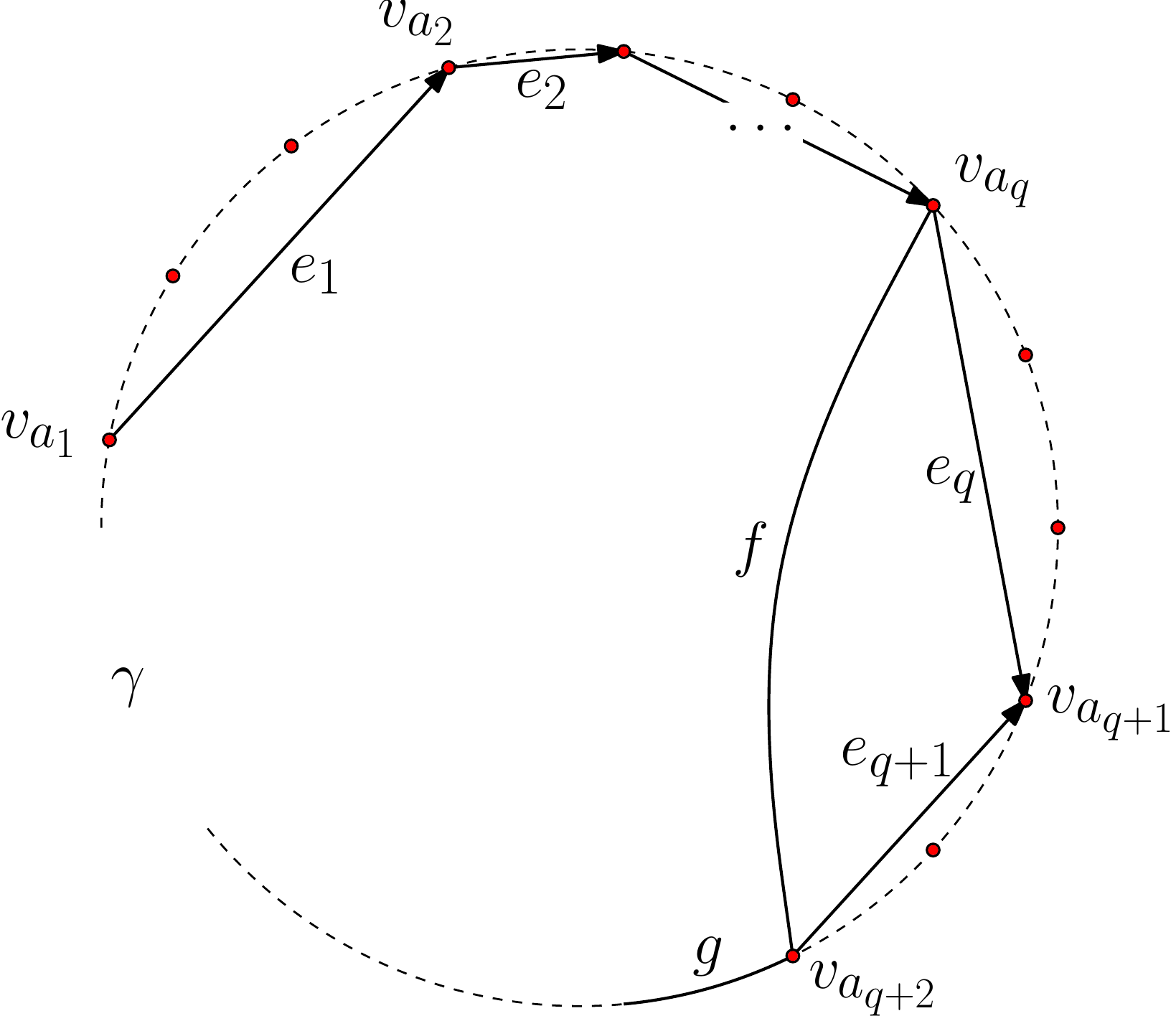}
\end{center}
\caption{Illustration for the proof of Lemma~\ref{lem:color}.}
 \label{fig:chicolor}
\end{figure}

To prove that $\arb_p(G) \le P_p(\mtrdens{\frac{p-1}{2}}(G))$ for some polynomial $P_p$,
it suffices to find a fraternal completion $\mathfrak f$ of an orientation $\vG$ of $G$ of depth $p$ so that each edge $e$ of $G$ is in
conflicts
with at most $P_p(\mtrdens{\frac{p-1}{2}}(G))$ other edges.
We shall see that this problem can be reduced to finding a fraternal completion with bounded in-degrees. Toward this end, let
$$C({\mathfrak f})=\sum \left\{\prod_j \md(\vH_{i_j}):{\sum_j i_j<{\rm depth}(\mathfrak f)}\right\}.$$

\begin{lemma}
\label{lem:din}
For every arc $e$ of $\vG$ we have
$$\bigl|\{ f\in E_{\mathfrak f}: f\succeq e\}\bigr| \leq C({\mathfrak f}).$$
\end{lemma}
\begin{proof}
For each $f \succeq e$, there is a sequence of arcs $g_1,g_2,\ldots,g_t$ in $\mathfrak f$ such that $g_1=e$, $g_t=f$ and $g_i$ covers $g_{i-1}$ in $\prec$.
So it suffices to show that for any   $g\succeq e$ with $w(f)=i$, for any $j > i$, there are at most $\md(\vH_{j-i})$ arcs $g'$ with $w(g')=j$ such that
$g'$ covers $g$ in $\prec$. This is so because the number of such arcs $g'$  is at most  the number of arcs $f'$
such that $\kappa(g')\in\{(g,f'),(f',g)\}$ and hence at most the number of arcs $f'$ with  ${\rm head}(f') = {\rm head}(g)$ in $\vH_{j-i}$, which is at most $\md(\vH_{j-i})$.
\end{proof}

\begin{lemma}
\label{lem:conflict}
For every arc $e_2$ of $\vG$ there exist at most $3pC({\mathfrak f})(\max(2,\md(\vH_p))^a$ arcs $e_1$ of $\vG$ such that $(e_1,e_2)$ is a conflict.
\end{lemma}
\begin{proof}
According to Lemma~\ref{lem:din} there exists at most $C({\mathfrak f})$ arcs $f_2\in \bigcup_{i=1}^p E_i$ such that $f_2\succeq e_2$.
Given an arc $f_2$ of $\vH_{\leq p}$ there exist at most $(\md(\vH_p)-1)(1+\dots+\md(\vH_p)^{p-1})=\md(\vH_p)^{p}-1$ arcs $f_1$ such that there exists in $\vH_{\leq p}$ a directed path of
length at most $p$ starting with $f_1$ and ending at the head of $f_2$. Similarly we have at most $1+\dots+\md(\vH_p)^{p}$ arcs $f_1$ such that there exists in $\vH_p$ a directed path of
length at most $p$ starting with $f_1$ and ending at the tail of $f_2$. Hence for each $f_2$ we have at most
$3\max(2,\md(\vH_p)^{p}$ possibilities for $f_1$. As there are $|W(f_1)|\leq a$ arcs $e_1$ such that $f_1\succeq e_1$, we conclude.
\end{proof}

\begin{lemma}
\label{lem:ecolor0}
Let $Q_1(X),\dots,Q_p(X)$ be polynomials, and let $P_p$ be the polynomial defined by:
\begin{equation}
P_p(X)=6p(2+Q_p(X))^p\biggl(\sum_{\sum i_j<p}\prod_j Q_{i_j}(X)\biggr)+1.
\end{equation}
Let $G$ be a multigraph with a fraternal completion $\mathfrak f$ of depth $p$ such that
\begin{equation}
\forall 1\leq i\leq p,\qquad \md(\vH_i)\leq Q_i(\mtrdens{(p-1)/2}(G)).
\end{equation}
Then
\begin{equation}
\arb_p(G)\leq P_p(\mtrdens{\tfrac{p-2}{2}}(G)).
\end{equation}
(Of course, the polynomial $P_p$ depends on polynomials $Q_1,\dots,Q_p$. This dependence will be clear from the context.)
\end{lemma}
\begin{proof}
According to Lemma~\ref{lem:conflict} and Lemma~\ref{lem:color}, $G$ has
a colouration
of its edges by at most $P_p(\mtrdens{(p-1)/2}(G))$ colours such that every cycle $\gamma$ gets at least $\min(p+1,|\gamma|)$ colours, that is
\begin{equation*}
\arb_p(G)\leq P_p(\mtrdens{\tfrac{p-1}{2}}(G)).
\end{equation*}
\end{proof}

By Lemma \ref{lem:ecolor0}, to prove that for some polynomial $P_p$,  $\arb_p(G) \le P_p(\mtrdens{\frac{p-1}{2}}(G))$,
it suffices to prove that one can find a fraternal completion $\mathfrak f$ of an orientation $\vG$ of $G$ of depth $p$ so that $\md(\vH_i)$ is bounded by
some polynomial function $Q_i$ of $\mtrdens{\frac{p-1}{2}}(G)$. The construction of such a fraternal completion $\mathfrak f$
is easy: Let $H_i$ be the underlying graph of $\vH_i$. By definition, $H_1=G$ and for $i=1,2, \ldots, p-1$, $H_{i+1}$ is uniquely determined by $\vH_j$ for $j=1,2, \ldots, i$.
For $i=1,2,\ldots, p$, we orient the edges of $H_i$ to obtain $\vH_i$ so that $\md(\vH_i) = \lceil \mtrdens{0}(H_i) \rceil$. This defines a fraternal completion $\mathfrak f$ of an orientation $\vG$ of $G$ of depth $p$. In the following, we shall show that for this fraternal completion $\mathfrak f$,  $\md(\vH_i)$ is bounded by
some polynomial function of $\mtrdens{\frac{p-1}{2}}(G)$.

For $i=1,2, \ldots,p$, let $T_i$ be the 
$(i-1)$-subdivision of 
the underlying graph $H_i$ of
$\vH_i$. Hence
$$H_i\in T_i\mshtm (\tfrac{i-1}{2}).$$
In particular, $H_1=T_1=G$.

For integer $m$, let $G\bullet m$ be the multigraph with vertex set $V(G)\times [m]$ where
$\{(x,i),(y,j)\}$ is an edge of $G\bullet m$ of multiplicity $k$ if and only if
$\{x,y\}$ is an edge of $G$ of multiplicity $k$. In the following section, we shall prove the following two lemmas.

\begin{lemma}
\label{lem:mlex}
Let $G$ be a multigraph, let $m$ be a positive integer and let $r$ be a positive half-integer. Then
\begin{equation}
\label{eq:mlex}
\mtrdens{r}(G\bullet m)\leq (2r(m-1)+1)\mtrdens{r}(G)+m^2\mtrdens{0}(G)+m-1.
\end{equation}
\end{lemma}

\begin{lemma}
\label{lem:fratemb}
For every integer $p \geq 2$,  there is a polynomial $N_p$ such that $T_p$ is a subgraph of $G \bullet N_p(\mtrdens{\frac{p-2}{2}}(G))$.
\end{lemma}

From these lemmas will then follow the main result of this paper:

\begin{theorem}
\label{thm:ecolor}
For each integer $p$ there exists a polynomial $P_p$ such that
for every multigraph $G$ it holds
$$\left( \mtrdens{\frac{p-1}{2}}(G)\right)^{1/p} \leq  \arb_p(G)  \leq P_p(\mtrdens{\frac{p-1}{2}}(G)).$$

In particular, $\arb_p$ and $\mtrdens{\tfrac{p-1}{2}}$ are two polynomially equivalent multigraph invariants.
\end{theorem}
\begin{proof}
The inequality
$\left( \mtrdens{\frac{p-1}{2}}(G)\right)^{1/p} \leq  \arb_p(G)$
follows from Lemma~\ref{lem:lowbound}.
By the argument above,  Lemma~\ref{lem:mlex} and Lemma~\ref{lem:fratemb} imply that $\md(\vH_i)$ is bounded by
some polynomial function of $\mtrdens{\frac{p-1}{2}}(G)$, and hence, by Lemma~\ref{lem:ecolor0}, there exists
a polynomial $P_p$ such that $\arb_p(G)  \leq P_p(\mtrdens{\frac{p-1}{2}}(G))$.
\end{proof}

\section{Proofs of Lemmas \ref{lem:mlex}  and \ref{lem:fratemb}}
\label{sec:5}

%

{\bf Proof of Lemma \ref{lem:mlex}}. 
The vertices of $K\bullet m$ are the pairs $(v,i)$, $v$ a vertex of $G$ and $1\leq i\leq m$.
The every vertex $v$ of $G$, we say that $(v,i)$ and $(v,j)$ are {\em twins} in $G\bullet m$ and we denote by $\pi$ the projection of $G\bullet m$ into $G$ which maps $(v,i)$ to $v$.

Let $S$ be a subgraph of $G \bullet m$ which is $(\leq 2r)$-subdivision of a multigraph $H$ such that
$$
\mtrdens{r}(G\bullet m)=\frac{\|H\|}{|H|}.
$$
Choose $S$ with the minimal number of vertices.

A path of $S$ corresponding to an edge of $H$ is called a {\em branch}. The vertices of $S$ corresponding to vertices of $H$ are called {\em principal vertices}. The other vertices of $S$ are {\em subdivision vertices}.

Let $S_0$ be the graph obtained from $S$ by deleting all the branches which are not subdivided
 and let $H_0$ be the corresponding subgraph of $H$ ($S_0$ is a $(\leq 2r)$-subdivision of $H_0$ and every branch of $S_0$ is a path of length at least $2$).
Then we have $\|H\|\leq \|H_0\|+m^2\|G[\pi(V(H))]\|$ hence
\begin{align*}
\frac{\|H\|}{|H|}&\leq \frac{\|H_0\|}{|H|}+m^2\frac{\|G[\pi(V(H))]\|}{|\pi(V(H))|}.
\intertext{Thus}
\mtrdens{r}(G\bullet m)&\leq \frac{\|H_0\|}{|H_0|}+m^2\mtrdens{0}(G).
\end{align*}
First notice that no branch of $S(H)$ contains two twin vertices, except if the branch is a single edge path linking two twin vertices (otherwise we can shorten the branch without changing $\|H\|$ and $|H|$, see Fig~\ref{fig:lexi0}).

\begin{figure}[h!t]
\includegraphics[width=\textwidth]{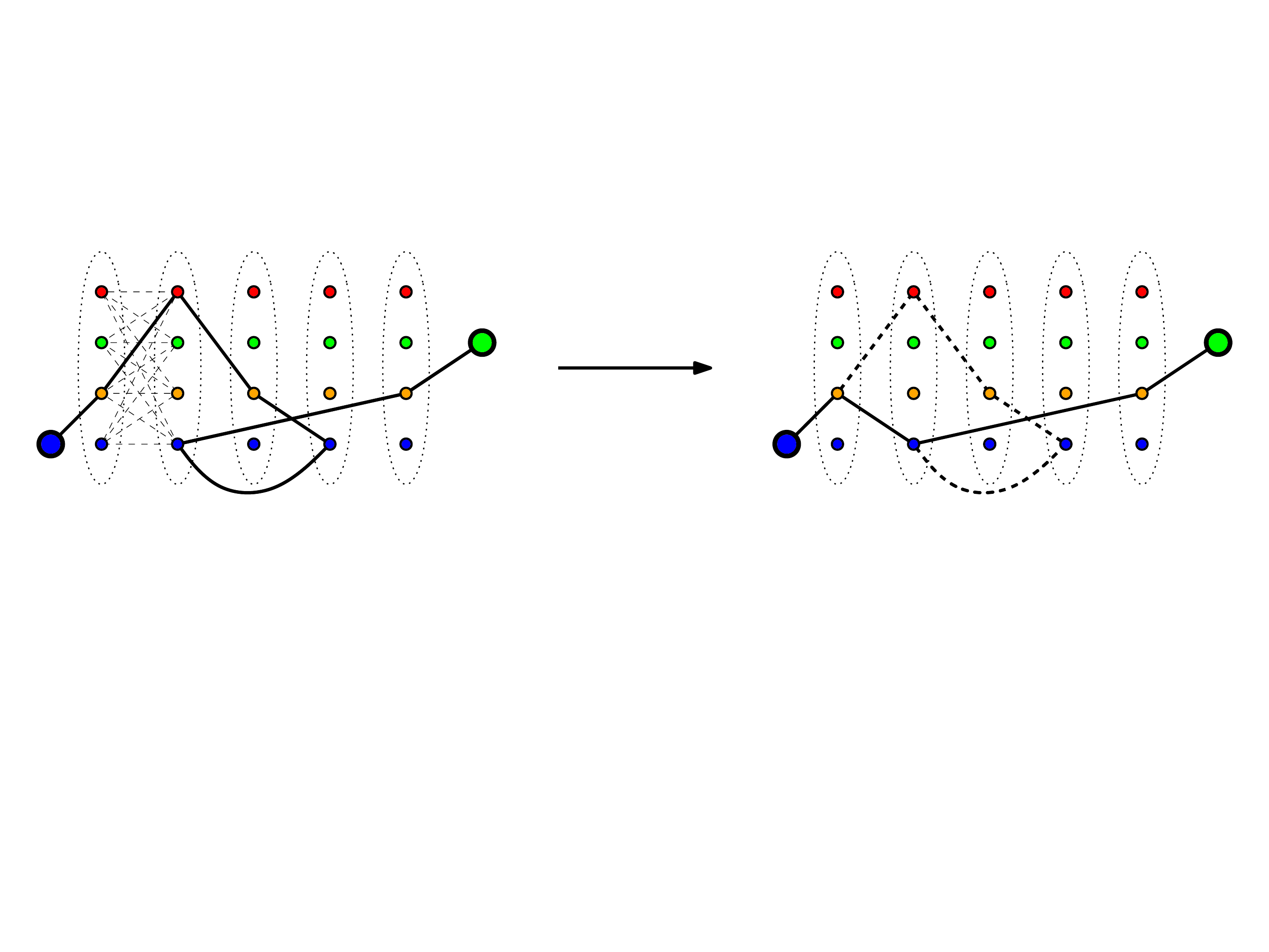}
\caption{If a branch contains twin vertices, we shorten it.}
 \label{fig:lexi0}
\end{figure}

We define the multigraph $H_1$ and its $(\leq 2r)$-subdivision $S_1$ by the following procedure:
Start with $H_1=H_0$ and $S_1=S_0$. Then, for each subdivision vertex $v\in S_1$ having a twin which is a principal vertex of $S_1$, delete the branch of $S_1$ containing $v$ and the corresponding edge of $H_1$.
In this way, we delete at most $(m-1)|H_0|$ edges of $H_1$.  Thus $\frac{\|H_1\|}{|H_1|}\geq\frac{\|H_0\|}{|H_0|}-(m-1)$ and
$S_1$ is such that no subdivision vertex is a twin of a principal vertex.

Given $H_1$ we construct the conflict graph $C$ of $H_1$ as follows: the vertex set of $C$ is the edge set of $H_1$ and the edges of $C$ are the pairs of edges $\{e_1,e_2\}$ such that one of the subdivision vertices of the branch corresponding to $e_1$ is a twin of one of the subdivision vertex of the branch corresponding to $e_2$. Note that graph $C$ has maximum degree at most $2r(m-1)$ hence it is $(2r(m-1)+1)$-colourable.
Let $H_2$ be the subgraph of $H_1$ induced by  a monochromatic set of edges of $H_1$ of size at least $\frac{\|H_1\|}{2r(m-1)+1}$.
So $\frac{\|H_1\|}{|H_1|} \le (2r(m-1)+1)\frac{\|H_2\|}{|H_2|}$.

Let  $S_2$ be the corresponding subgraph of $S_1$.  If $v$ is a principal vertex of $S_2$, then two edges incident to $v$ cannot have their other endpoints equal or twins (because of the colouration).
Let $S_3=\pi(S_2)$ be the projection of $S_2$ on $G$.  Because of the above colouration,
$S_3$ is a ($\le 2r$) subdivision of $H_2$. Hence  $$\mtrdens{r}(G)\geq \frac{\|H_2\|}{|H_2|} \ge \frac{\|H_1\|}{(2r(m-1)+1)|H_1|} \ge \frac{\|H_0\|}{(2r(m-1)+1)|H_0|} -\frac{m-1}{2r(m-1)+1} $$ and the result follows.
\qed

{\bf Proof of Lemma \ref{lem:fratemb}}

Let $\mathfrak f=((E_1,\dots,E_p),w, \kappa)$ be a fraternal
completion of $\vG$ of depth $p$ constructed in such a way that for $i=1,2, \ldots, p$,  $\md(\vH_i) = \mtrdens{0}(H_i)$.

\begin{figure}[h!t]%
\includegraphics[width=\columnwidth]{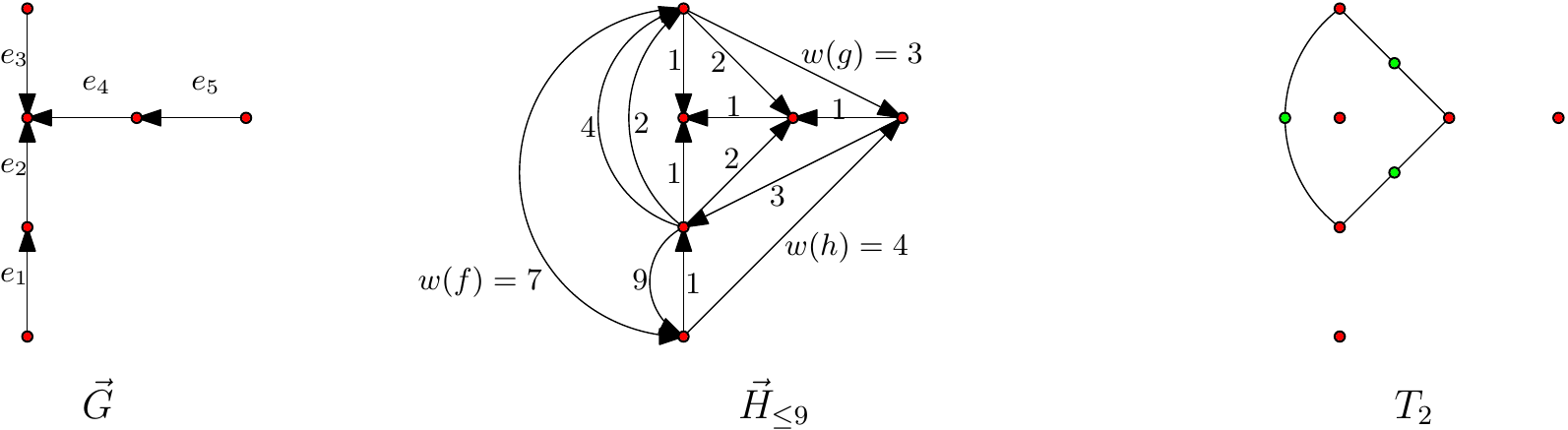}%
\caption{Example of graphs defined by a fraternal completions. On the left, the graph $\vG$.
In the middle, arcs of a fraternal completion of depth $9$ form the multigraph $\vH_{\leq 9}$; here, $\kappa(f)=(g,h)$; the walk
$W(h)=(e_1 e_2 \overline{e}_4 \overline{e}_5)$  associated to $h$ is simple but the walk $W(f)=(e_3 \overline{e}_4 \overline{e}_5 e_5 e_4 \overline{e}_2 \overline{e}_1)$ associated to $f$ is not.
 On the right, the $1$-subdivision of the arcs in $E_2$ define the undirected multigraph $T_2$.}
\label{fig:frat}%
\end{figure}

In the following, we shall prove that for $2 \leq a\leq p$ there is a polynomial $N_a$ such that
the graph $T_a$ can be injectively embedded into a blowing $G \bullet N_a(\mtrdens{\frac{a-2}{2}}(G))$ of $G$.
Observe that if this is true for $a=2, \ldots, i$, then by using Lemma \ref{lem:mlex}, we can conclude that there is a polynomial $P_i$ such that $\mtrdens{0}(H_i) \le P_i(\mtrdens{\frac{i-1}{2}}(G))$.

By definition, for $a \ge 1$, $T_a$ is obtained from the empty graph on $V(G)$ by adding, for each arc $e=xy$ in $E_a$, an induced path of length $a$ connecting $x$ and $y$. Each arc $e=xy$ in $E_a$ corresponds to a walk $W(e)$ in $G$ of length $a$ from $x$ to $y$. By sending the induced $x$-$y$-path of length $a$ in $T_a$ to the corresponding walk in $G$ connecting $x$ and $y$, we obtain a homomorphism, say $f$, from $T_a$ to $G$. However, $f$ is not an embedding, as many vertices of $T_a$ may have the same image. 
Indeed, $V(T_a)$ is the union of $V(G)$ and a set of $|E(T_a)|(a-1)$ added vertices which are the interior vertices of the walks $W(e)$. If for some integer $m$, each vertex $v$ of $G$ is contained in at most $m-1$ of the walks $W(e)$ as an interior vertex, then $T_a$ embeds into $G \bullet m$, as we can assign to each vertex  in $f^{-1}(v)$ a distinct vertex of $\{v\} \times [m]$ in $G \bullet m$  as its image. Let us consider the 
case $a=2$. By our construction, $\vG$ has $\md(\vG) = \lceil \mtrdens{0}(G)\rceil$ Each arc $e=xy$ in $E_2$ corresponds to a walk of the form $(x,v,y)$, where $(x,v)$ and $(y,v)$ are arcs of $\vG$. Let $d= \lceil \mtrdens{0}(G) \rceil = \md(\vG)$. Then   $v$ is an interior vertex of a walk $W(e)$ for $e=xy$ in $E_2$ if and only if $(x,v)$ and $(y,v)$ are arcs of $\vG$. Hence $v$ is contained in at most ${ d \choose 2}$ walks $W(e)$ as a an interior vertex.  Therefore for $N_2(x) = {x +1 \choose 2}+1$, $T_2$ embeds into $G \bullet N_2(\mtrdens{0}(G))$.

Assume now that the polynomial $N_i$ is defined for $i=2,\ldots, a-1$ and $T_i$ embeds into $G \bullet N_i(\mtrdens{\frac{i-2}{2}}(G))$. 
Each arc $e =xy \in E_a$ corresponds to two arcs $g \in E_i$ and $g' \in E_j$ with $i+j =a$, with $g=xz$ and $g'=yz$ for some $z$. A vertex $v$ is an interior vertex of the walk $W(e)$ if and only if either
$v=z$ or $v$ is an interior vertex of $W(g)$ or $W(g')$. By our definition of the fraternal completion, $v$ has in-degree  at most $\md(\vH_i) = \lceil \mtrdens{0}(H_i) \rceil$ in $\vH_i$. By induction hypothesis and the observation above, $\mtrdens{0}(H_i) \le P_i(\mtrdens{\frac{i-1}{2}}(G))$ for some polynomial $P_i$. This implies that for some polynomial $Q$, each vertex $v$ has in-degree at most $d=Q(\mtrdens{\frac{a-2}{2}}(G))$ in $\vH_{\le (a)}$. So for each vertex $v$ of $G$, there are at most  ${d \choose 2}$ pairs of arcs $g, g'$ for which there is an edge $e=xy$ in $E_a$ with $k(e)=(g,g')$. 

For an edge $ e=xy \in E_a$, we say a vertex $v$ is the {\em transfer vertex} of $W(e)$ if $k(e)=(g,g')$, $g=(x,v)$ and $g'=(y,v)$. An interior vertex of $W(e)$ is either a transfer vertex of $W(e)$ or an interior vertex of $g$ for some $ g \in E_i$ where $i \le a-1$. 

By induction hypothesis, there is a polynomial $Z$ such that a vertex $v$ appears at most $Z(\mtrdens{\frac{a-3}{2}}(G))$ times as an interior vertex of a walk $W(g)$ for $g \in \cup_{i=1}^{a-1}E_i$. For each $g \in E_i$ for some $i \le a-1$, there are at most $2d$ edges $e \in E_a$ such that $k(e)=(g, g')$ or $k(e)=(g', g)$. Therefore, each vertex $v$ appears at most ${d \choose 2} + 2d \times Z(\mtrdens{\frac{a-2}{2}}(G))$ times as an interior vertex of $W(e)$ for $e \in E_a$. As $d=Q(\mtrdens{\frac{a-2}{2}}(G))$ and $\mtrdens{\frac{a-3}{2}}(G) \le \mtrdens{\frac{a-2}{2}}(G)$, with $N_a(x) = Q(x)^2 + 2Q(x)Z(x)$,  each vertex appears at most $N_a(\mtrdens{\frac{a-2}{2}}(G))-1$ times as an interior vertex of $W(e)$ for some $e \in E_a$. Therefore with $m = N_a(\mtrdens{\frac{a-2}{2}}(G))$,   $T_a$ embeds into $G \bullet m$. This completes the proof of Lemma \ref{lem:fratemb}.
\qed

\section{The Dual Version}
\label{sec:6}
The problem addressed in this paper can be considered in the more general context of matroids:
\begin{problem}
Let $\mathcal M$ be a matroid and let $p$ be an integer. What is the minimum  number $\arb_p(\mathcal M)$ needed to colour the elements of $\mathcal M$ in such a way that each circuit $\gamma$ gets at least $\min(|\gamma|,p+1)$ colours?
\end{problem}

It would be interesting to find a natural class of matroids for which $\arb^\star_p(\mathcal M)$ is uniformly bounded.
For graphs this leads to the following problem: 
\begin{problem}
Let $G$ be a graph and let $p$ be an integer. What is the minimum integer $N=\arb^\star_p(G)$ such that the edge set of $G$ may be coloured using $N$-colours in such a way that each cut $\omega$ gets at least $\min(|\omega|,p+1)$ colours?
\end{problem}

It is maybe interesting that the dual version of our problem may present different aspects.
The well known theorem of Erd{\H o}s \cite{ErdH1959} which asserts that there exists a graph of order at least $n$, girth at least $g$ and chromatic number at least $2N+1$. As the chromatic number of a graph is bounded by $\chi(G)\leq 2\arb(G)+1$ we get that there exist graphs with arbitrarily large girth and arboricity (hence arbitrarily large $\arb_p$). 
The notion dual to ``$G$ has girth at least $k$'' (i.e. every cycle of $G$ has length at least $k$) is ``$G$ is $k$-edge connected'' (i.e. every edge cut of $G$ has size at least $k$). 
However, there does not exist graphs with arbitrarily edge-connectivity and $\arb_p^\star$. Precisely:

\begin{proposition}
Let $G$ be a graph and let $p$ be an integer.
\begin{itemize}
    \item If $G$ is $(2p+2)$-edge connected then $\arb^\star_p(G)=p+1$;
    \item if $G$ is $(2p+1)$-edge connected then $\arb^\star_p(G)\leq
    (p+1)(2p+1)$, and there exists infinitely many $(2p+1)$-edge connected graphs
    such that $\arb^\star_p(G)\geq p+2$;
    \item $\arb^\star_p(G)$ is not bounded for $(2p)$-edge connected graphs. 
\end{itemize}
\end{proposition}
\begin{proof}
The first item is a consequence of \cite{K74} where it is proved that a $2n$-edge connected graph has at least $n$ pairwise
edge-disjoint spanning trees. It follows that if $G$ is $(2p+2)$-edge connected, it has at least $p+1$ edge-disjoint spanning trees $Y_1,\dots,Y_{p+1}$. Colour $i$ the edges of $Y_i$ and further colour $1$ the edges which are present in none of the $Y_i$'s. As each $Y_i$ is spanning, each cut meets all the $Y_i$'s thus gets $p+1$ colours. It follows that $\arb^\star_p(G)=p+1$.

The upper bound of the second item is similarly obtained by doubling each edge of $G$ (thus obtaining a $(4p+2)$-edge connected multigraph) and considering $2p+1$ edge-disjoint spanning trees of this new multigraph $G'$, and colouring each edge $e$ of $G$ by the set of (at most two) colours assigned to the two edges of $G'$ corresponding to $e$. The lower bound is obtained by considering non-complete $(2p+1)$-regular $(2p+1)$-edge connected graphs $G$: if $\arb_p(G)=p+1$ would hold then each colour class would include a spanning tree and hence $\|G\|\geq (p+1)(|G|-1)$ would hold.

The last item follows from the following construction. For integers $L,p$ let $C_L^{(p)}$ be the multigraph obtained from a cycle of length $L$ by replacing each edge by $p$ parallel edges (see Fig.~\ref{fig:mcycle}). The graph $C_L^{(p)}$ is $(2p)$-edge connected. However, $\arb_p(C_L^{(p)})\geq L^{1/p}$ as if each cut gets at least $p+1$ colours then
no two group of parallel edges can be coloured by the same set of colours.

\end{proof}

\begin{figure}[h]
	\centering
		\includegraphics[width=.3\textwidth]{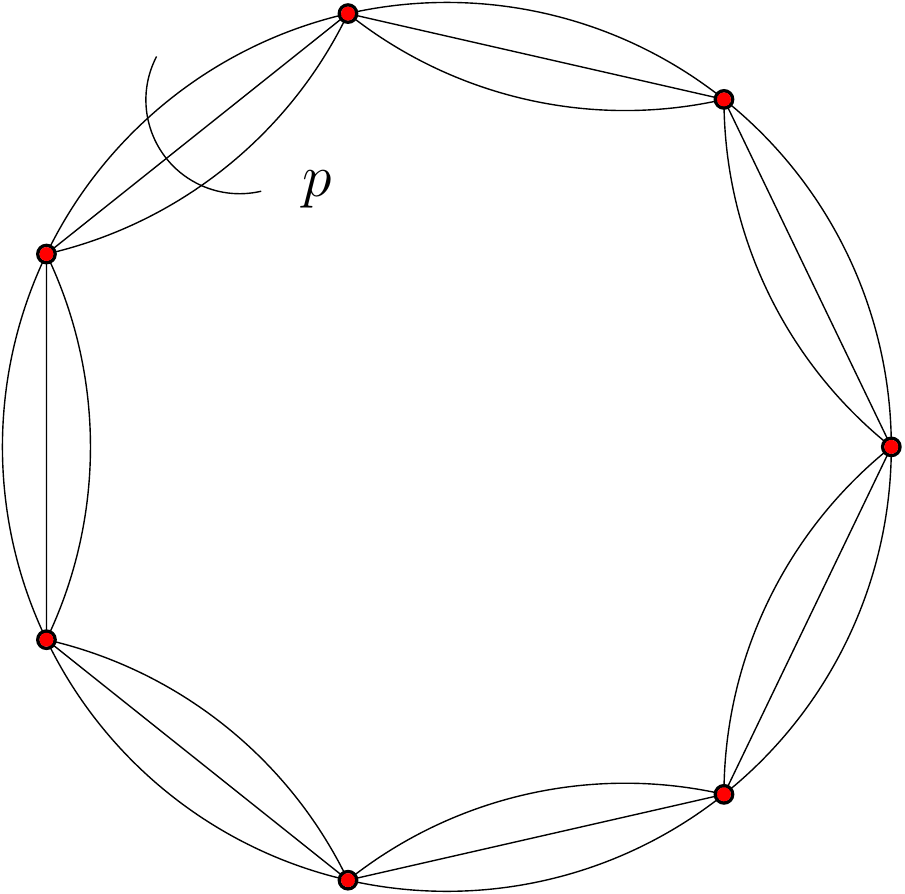}
	\caption{The multigraph $C_L^{(p)}$ is $(2p)$-edge connected and $\arb_p(C_L^{(p)})\geq L^{1/p}$}
	\label{fig:mcycle}
\end{figure}

\section*{References}
\providecommand{\noopsort}[1]{}\providecommand{\noopsort}[1]{}
\providecommand{\bysame}{\leavevmode\hbox to3em{\hrulefill}\thinspace}
\providecommand{\MR}{\relax\ifhmode\unskip\space\fi MR }
\providecommand{\MRhref}[2]{%
  \href{http://www.ams.org/mathscinet-getitem?mr=#1}{#2}
}
\providecommand{\href}[2]{#2}

\end{document}